\documentclass[12pt, a4paper]{amsart}%
\usepackage{amsmath}
\usepackage{amsfonts}%
\usepackage{mathtools} 
\usepackage{amssymb}
\usepackage{enumerate}
\usepackage{xypic} 

\newcommand{\field}[1]{\ensuremath{\mathbb{#1}}}

\newcommand{\bbc}{\field{C}}
\newcommand{\bbd}{\field{D}}

\newcommand{\bbn}{\field{N}}
\newcommand{\bbp}{\field{P}}

\newcommand{\bbz}{\field{Z}}

\newcommand{\mbi}{{\rm{i}}}

\newcommand{\mbK}{\mathbf{K}}

\newcommand{\mbo}{\mathbf{o}}

\renewcommand{\Im}{\operatorname{Im}}

\renewcommand{\Re}{\operatorname{Re}}

\newcommand{\diag}{\operatorname{diag}}
\newcommand{\cn}{\operatorname{cn}}
\newcommand{\dn}{\operatorname{dn}}
\newcommand{\sn}{\operatorname{sn}}

\providecommand{\U}[1]{\protect\rule{.1in}{.1in}}
\newtheorem{theorem}{Theorem}
\theoremstyle{plain}

\newtheorem{definition}[theorem]{Definition}

\newtheorem{proposition}[theorem]{Proposition}

\theoremstyle{definition}

\newtheorem{remark}[theorem]{Remark}

\numberwithin{equation}{section}
\numberwithin{theorem}{section}

\begin{document}
\title[Lax equations, Singularities and Riemann-Hilbert Problems]{Lax equations, Singularities and Riemann-Hilbert Problems}
\author[A.F. dos Santos]{Ant\'{o}nio F. dos Santos}
\address{Departamento de Matem\'atica, Instituto Superior T\'ecnico, Portugal}
\email{afsantos@math.ist.utl.pt}
\author[P.F. dos Santos]{Pedro  F. dos Santos}
\address{Departamento de Matem\'atica, Instituto Superior T\'ecnico, Portugal}
\email{pedro.f.santos@math.ist.utl.pt}

\
\maketitle
\begin{abstract}
The existence of singularities  of the solution for a class of Lax equations 
is investigated using a development of the factorization method first proposed by Semenov-Tian-Shansky and Reymann \cite{FAA},
\cite{EMS}.
It is shown that the existence of a singularity at  a point $t=t_i$ is directly related to the property that
the kernel of a certain Toeplitz operator (whose symbol depends on $t$) be non-trivial. The investigation of this question
involves the factorization on a Riemann surface of a scalar function closely related to the above-mentioned operator. 
An example  is presented and  the set of singularities is shown to coincide 
with the set obtained by classical methods. This comparison involves relating the two
Riemann surfaces associated  to the system by these methods.
\end{abstract}

\section{Introduction}
\label{sect.1}
In this paper we investigate the existence of singularities of the solutions of Lax equations for a class 
of  equations that applies to most finite-dimensional dynamical systems such as \emph{e.g.} classical
tops (see \emph{e.g.} \cite{AU96}, \cite{EMS}, \cite{STS00}). To that end we consider the time variable $t$ to be a complex
variable and determine the singularities of the solution in the complex plane. This is tied to the question of global existence of  solutions
for real $t$ as  the non-existence of singularities for real $t$ implies global existence of the solution. Also, it is likely   that full knowledge
of the location of complex singularities may eventually give more insight into the dynamics of the system.

Our approach  is a development of the factorization method first proposed by Semenov-Tian-Shansky and Reymann \cite{FAA},
\cite{EMS} which  in turn may be seen  as a generalization of the AKS (Adler-Kostant-Symes) theorem that  applies to finite dimensional algebras \cite{AMV}.

To the best of our knowledge the first application of this method in the setting of an infinite-dimensional algebra
appeared in \cite{CdSdS07}, which focused on a restricted class of Lax equations. The absence in the literature of more fully computed examples of
application of this method is probably due to the fact that it involves the factorization of a continuous function on a contour
in a Riemann surface (for a general treatment of this problem see \cite{CdSdS08}).

Considering  $t$ as a complex variable and extending the class of Lax equations requires a new analysis of the  results of \cite{CdSdS07}, where we were able to avoid
making some delicate assumptions like the differentiability of the factors  in the Wiener-Hopft factorization (see definition below) of the matrix function $\exp(tL_0)$, where $L_0$ is the value of the
Lax matrix $L_t$ at $t=0$. 
This is done in Section~\ref{sect.2} and continued in Section~\ref{sect.3} for the question of location of the singularities of the solution.
Our approach makes the treatment fully rigorous and, in our view, is crucial for the treatment in the context of a complex $t$.

The main result is Theorem~\ref{thm.3.1}, which we state next.  For this we note that the  space
$\left[C_\mu\left(S^1\right)  \right]^n$ of H\"{o}lder continuous $n\times n$ 
matrix-valued functions has a direct sum decomposition
\begin{equation}
\label{eq.1.2}
\left[C_\mu\left(S^1\right)  \right]^n = \left[C_\mu^+  \right]^n \oplus \left[C_\mu^-  \right]^n_0,
\end{equation}
where $\left[C_\mu^+  \right]^n$ is the subspace \footnote{ We have  omitted  $S^1$ to simplify notation} of functions having an analytic extension to the unit disc $\bbd$ and
$\left[C_\mu^- \right]^n_0$ is the subspace  of functions
admitting analytic extensions  to $\bbc\setminus \overline{\bbd}$
that vanish at infinity. In what follows  $\left[C_\mu^- \right]^n=\left[C_\mu^- \right]^n_0\oplus\bbc^n$. 

Consider the   Toeplitz operator 
\begin{equation}
\label{eq.1.1}
T_G = P^+GI_+\colon \left[C_\mu^+  \right]^n\to \left[C_\mu^+ \right]^n
\end{equation}
where $G=\exp(tL_0)$, $I_+$ is the identity operator on $\left[C_\mu^+ \right]^n$ and $P^+$ is the projection of 
$\left[C_\mu  \right]^n$ onto $\left[C_\mu^+  \right]^n$ associated to the decomposition
\eqref{eq.1.1}. Theorem~\ref{thm.3.1} states that the Lax equation,
\begin{align}
\label{eq.1.3}
\frac{dL_t}{dt} = \left[L_t^+ + A_0, L_t  \right], 
\end{align}
has a solution in a neighbourhood of the point $t_i$ \emph{iff} $T_G$ is injective at the point $t_i$.
In the above equation $A_0 = P_0L_t$ where $P_0\colon \left[ C_\mu(S^1)\right]^{n\times n}\to \bbc^{n\times n}$
is a bounded linear operator.

In the calculation of the singularities we need the notion of a Wiener-Hopf factorization. 
Let  $G\colon S^1\to \left[ C_\mu\right]^{n\times n}$. We say that $G$ possesses a \emph{Wiener-Hopf factorization}
(also called Riemann-Hilbert factorization and Birkhoff factorization \cite{G97}, \cite{PS86}) if $G$ can  be 
represented in the form 
\begin{equation}
\label{eq.1.4}
G=G_- D G_+,
\end{equation}
where $G_\pm$ and their inverses belong to $\left[C_\mu^\pm \right]^{n\times n}$ and $D=\diag(r^{k_1},\dotsc, r^{k_n})$ with 
$k_1\geq k_2\geq\dotsb k_n$ and $r$ is a rational function with a zero in $\bbd$ and a pole in $\bbc\setminus\overline{\bbd}$.
The factorization is said to be \emph{canonical} if $D=I_n$,  where $I_n$ denotes the identity matrix. The above definition applies to 
functions belonging to other spaces (see \emph{e.g.} \cite{ClGo81}). In  $[C_\mu]^{n\times n}$ $G$ possesses a Wiener-Hopf
factorization \eqref{eq.1.3} \emph{iff} $G$ is invertible on $S^1$. We recall from \cite{ClGo81} that the operator $T_G$ is invertible
\emph{iff} the factorization \eqref{eq.1.4} is canonical. This is the basic result from operatior theory that will be used to locate the 
singularities of the solution of the Lax equations. A direct consequence of Theorem~\ref{thm.3.1} is Proposition~\ref{prop.3.2} 
which states that the solution of equation \eqref{eq.1.3}
has  a singularity at $t=t_i$  \emph{iff} the Riemann-Hilbert problem
\begin{align*}
G \Phi^+ = \Phi^- && \left( \Phi^+ \in  \left[C_\mu^+ \right]^n, \, \Phi^-\in  \left[C_\mu^-  \right]^n_0\right)
\end{align*}
has a nontrivial solution at $t=t_i$. This is equivalent to saying that the Wiener-Hopf factorization of $G$ is noncanonical  
for $t=t_i$.

The paper ends with an example of a dynamical system that belongs to the standard Lax class considered  in \cite{CdSdS07}. For this example it is possible to obtain the solution to the Lax equation by classical methods (integration of the system of ordinary 
differential equations) and thus obtaining its set of singularities. This  enables us to compare it with the set of singularities derived by our method.   A rather
interesting point is that the classical approach and the Lax  equation one lead to different Riemann surfaces. The two surfaces are closely related, as shown in Proposition~\ref{prop.4.2}, but
the fact that they are different led us to  derive several intermediate results in order  to show that the sets of singularities obtained by 
the two approaches coincide. 

The study of the example given in Section~4 takes a large part of the paper but we believe that it not only illustrates the theory that
we present here but also sheds some light into the relation between the classical methods and those based on the Lax equation
- a point that may be obscure  in the study of other finite-dimensional integrable systems, for example, in the study of some classical 
tops.

\section{Lax equation and Riemann-Hilbert problems}
\label{sect.2}

In this section we generalize the results of \cite[\S 2]{CdSdS07} replacing a neighborhood of the origin in the real  variable $t$ by 
a neighborhood of the origin in the now complex variable $t$ (for what follows it is necessary to consider $t$ as complex variable)
and extending the class of equations considered.
In \cite{CdSdS07} we studied a class of Lax equations of the form
\begin{equation}
\label{eq.2.0}
\frac{dL_t}{dt}= \left[L_t^+, L_t\right]
\end{equation}
where the dynamical variables $L_t^+$, $L_t$ depend on  a parameter $\lambda$ varying on the unit circle $S^1$, $L_t$ is a matrix-valued
Laurent polynomial in $\lambda$ and $L_t^+$ is the part of $L_t$ analytic in the unit disc $\bbd$. In \cite{CdSdS07} we called the above class
the \emph{standard Lax class} (it includes \emph{e.g.} a special case of  the Lagrange top).
In this paper we study a class of Lax equations more general than the above one. It includes most finite-dimensional integrable systems. We write
the equations of this class in the form 
\begin{equation}
\label{eq.2.1}
\frac{dL_t}{dt}= \left[L_t^+ + A_0, L_t\right]
\end{equation}
where $L_t^+$ is defined as above and $A_0 = P_0L_t$, with $P_0$ being a bounded linear operator from the space $\left[C^1(\Omega)\right]^{n\times n}$ of matrix-valued H\"{o}lder functions on $S^1$ to the space $\bbc^{n\times n}$
of constant matrix functions on $S^1$ (depending on $t$  as a parameter).

To state the first result in a rigorous way we need the definition that follows 

\begin{definition} 
\label{def.2.1}
Let  $\left[C^1(\Omega)\right]^{n\times n}$ be the space of continuously differentiable matrix functions with respect to $t$ in a region $\Omega\subset\bbc$
and define $L_t(\lambda)\in\left[C^1(\Omega)\right]^{n\times n}$ to be a Laurent polynomial of the form
\begin{equation}
\label{eq.2.2}
L_t(\lambda) = \sum_{k=-m}^1 L_t^{(k)}\lambda^k \qquad (m\in\bbn, \lambda \in S^1),
\end{equation}
where $L_t^{(k)}\in \bbc^{n\times n}$.
This gives for $L_t^+(\lambda)$ the expression
\begin{equation}
\label{eq.2.3}
L_t^+(\lambda) = \sum_{k=0}^1 L_t^{(k)}\lambda^k.
\end{equation}
\end{definition}

\begin{remark}
\label{rmk.22}
In the case of the standard Lax class ($A_0=0$ in \eqref{eq.2.1}) equation  \eqref{eq.2.1} together with formulas \eqref{eq.2.2} and \eqref{eq.2.3} 
imply that $L_t^{(1)}$ is a constant of the dynamics.
\end{remark}

We can now state our first result which is a generalization of \cite[Theorem 2.3]{CdSdS07} extending the applicability of known formulas
(see \emph{e.g.} \cite{EMS}) for $L_t$.

\begin{theorem}
\label{thm.2.3}
Let $L_t$ be an $n\times n$ matrix-valued function satisfying the Lax equation~\eqref{eq.2.1} in a simply-connected region $\Omega$
containing the origin (in the variable $t$). Then $L_t$ is given in the region $\Omega$ by the formulas
\begin{equation}
\label{eq.2.4}
L_t = \widetilde{G}_+ L_0 \widetilde{G}_+^{-1} = \widetilde{G}_-^{-1} L_0 \widetilde{G}_-
\end{equation}
where $L_0 = L_t|_{t=0}$ and $\widetilde{G}_+, \widetilde{G}_-$ satisfy in $\Omega$ the linear differential equations
\begin{align}
\label{eq.2.5}
\frac{d\widetilde{G}_+}{dt} = (L_t^+ + A_0)\widetilde{G}_+ && \frac{d\widetilde{G}_-}{dt} = \widetilde{G}_-(L_t^- - A_0)
\end{align}
subject to the initial conditions $\widetilde{G}_+|_{t=0} = \widetilde{G}_-|_{t=0} = I_n$, 
where $I_n$ is the identity $n\times n$ matrix.
\end{theorem}
\begin{proof}
The proof goes along the same lines as the proof  of \cite[Theorem 2.3]{CdSdS07} with $L_t^\pm$ replaced by $L_t^\pm\pm A_0$.
We note only that the condition on the connectivity
of $\Omega$ is needed to ensure that equations~\eqref{eq.2.5} have well defined solutions throughout $\Omega$.
\end{proof}

\begin{proposition}
\label{prop.2.4}
If the singularities of $L_t$ with respect to $t$ are isolated then there exists a simply connected region $\Omega=\bbc\setminus B_s$ where $B_s$ is the union of two cuts joining the singularities.
\end{proposition}
\begin{proof}
If the singularities are isolated we can denote them by $t_n$ with $n\in\bbz$. 
Furthermore we can enumerate them in \emph{lexicographic order} ($t\leq t' $ \emph{iff} $\Re t < \Re t'$ or $\Re t=\Re t'$ and $\Im t \leq \Im t'$).
Consider two consecutive points of this sequence, say $t_k, t_{k+1}$. Define one cut ($B_{s^+}$) as the union of line segments 
$[t_{r-1}, t_r]$ for $r<k$. Similarly, define the second cut ($B_{s^-}$) as the union of segments $[t_r, t_{r+1}]$ for $r \geq k$.
Then $\bbc\setminus (B_{s^+}\cup B_{s^-})$ is simply connected.
\end{proof}

\begin{theorem}
\label{thm.2.5}
Let $G=\widetilde{G}_-\widetilde{G}_+$ where $\widetilde{G}_-, \widetilde{G}_+$ satisfy equations~\eqref{eq.2.5} in a simply connected region $\Omega$ containing the origin (in the variable $t$)
and the condition $G|_{t=0}=I_n$. Then
\begin{equation}
\label{eq.2.6}
G= \exp(tL_0).
\end{equation}
\end{theorem}
\begin{proof}
\begin{align*}
\frac{dG}{dt} &= \frac{d\widetilde{G}_-}{dt}\,\widetilde{G}_+  + \widetilde{G}_-\frac{d\widetilde{G}_+}{dt}\\
                       & = \widetilde{G}_-\left[\widetilde{G}_-^{-1}\frac{d\widetilde{G}_-}{dt} + \frac{d\widetilde{G}_+}{dt} \widetilde{G}_+^{-1}\right] \widetilde{G}_+.
\end{align*}
From equations \eqref{eq.2.5}
\[
\widetilde{G}_-^{-1}\,\frac{d\widetilde{G}_-}{dt} + \frac{d\widetilde{G}_+}{dt} \widetilde{G}_+^{-1} = L_t^- - A_0 + L_t^++A_0 = L_t. 
\]
Then
\[
\frac{dG}{dt} = \widetilde{G}_- L_t \widetilde{G}_+ = L_0 G,
\]
where we have introduced the expression for $L_0$ resulting from the second of formulas \eqref{eq.2.4}, $L_0=\widetilde{G}_-L_t\widetilde{G}_-^{-1}$. Formula \eqref{eq.2.6} now follows from the above equation.
\end{proof}

\begin{theorem}
\label{thm.2.6}
\hfill
\begin{enumerate}[(i)]
\item The factorization of $G=\exp(tL_0)$, $G=\widetilde{G}_-\widetilde{G}_+$ is a canonical
Wiener-Hopf factorization in the region $\Omega$ of Theorem~\ref{thm.2.5}.
\item Let $G_-G_+$ be another Wiener-Hopf factorization. Then
\begin{align}
\label{eq.2.7}
\widetilde{G}_+ = G_+ F_t,&& \widetilde{G}_- = G_- F_t^{-1}.
\end{align} 
\end{enumerate}
\end{theorem}

\begin{proof}
\begin{enumerate}[(i)]
\item
Let $\widetilde{G}_-\widetilde{G}_+$ be the factorization of $G$ obtained in Theorem~\ref{thm.2.5}, \emph{i.e.}
\[
G = \exp\left( tL_0\right) = \widetilde{G}_- \widetilde{G}_+.
\]
This is a canonical Wiener-Hopf factorization of $G$ in view of the properties of $\widetilde{G}_-,  \widetilde{G}_+$
resulting from equations \eqref{eq.2.5}.
\item
Let $G_-G_+$  be another (canonical) Wiener-Hopf factorization of $G$, obtained \emph{e.g.} by solving a Riemann-Hilbert
problem with coefficient $G$ ( $G\Phi^+=\Phi^-$). Then we have
\begin{equation}
\label{eq.2.8}
G= G_- G_+ = \widetilde{G}_- \widetilde{G}_+.
\end{equation}

In the above relation both factorizations have factors that with their inverses are bounded analytic in their domains of existence,
$G_-$, $G_+$, because it is assumed to be a Wiener-Hopf factorization and $\widetilde{G}_-$,  $\widetilde{G}_+$ because the factors are assumed to satisfy equations \eqref{eq.2.5}. 

From \eqref{eq.2.8} we have
\begin{equation}
\label{eq.2.9}
\widetilde{G}_-^{-1}  G_-= \widetilde{G}_+ G_+^{-1}
\end{equation}
which implies that both sides equal a constant in $\lambda$, but since we have a relation \eqref{eq.2.9} for each $t$, both sides must
equal a function of $t$, independent of $\lambda$. We write
\[
\widetilde{G}_-^{-1}  G_-= \widetilde{G}_+ G_+^{-1} = F_t
\]
\emph{i.e.}
\begin{align*}
\widetilde{G}_- = G_- F_t^{-1}, && \widetilde{G}_+ = F_t G_+.
\end{align*}
\end{enumerate}
\end{proof}

$F_t$ plays the role of a normalization factor at the point $t$ for $G_-$, $G_+$.

\begin{theorem}
\label{thm.2.7}
Let $F_t$ satisfy the linear differential equation
\begin{equation}
\label{eq.2.10}
\frac{dF_t}{dt} = A_0 F_t.
\end{equation}
Then the function $L_t$ given by
\begin{equation}
\label{eq.2.11}
L_t = F_t \hat{L}_t F_t^{-1},
\end{equation}
with $\hat{L}_t = G_+ L_0 G_+^{-1}$, satisfies the Lax equation \eqref{eq.2.1}. Here $G_+$ is as in Theorem~\ref{thm.2.6}
(as are  $G_-, G$).
\end{theorem}

\begin{proof}
We show that $L_t$ given by \eqref{eq.2.11} with $F_t$ satisfying \eqref{eq.2.10}  is a solution to equation \eqref{eq.2.1}. 
We have
\begin{align*}
\frac{dL_t}{dt} = A_0F_t \hat{L}_t F_t^{-1} + F_t \frac{d\hat{L}_t}{dt} F_t^{-1} - F_t \hat{L}_t F_t^{-1} A_0 F_t F_t^{-1}\\
	                = \left[ A_0, L_t \right]  + F_t  \frac{d\hat{L}_t}{dt}  F_t^{-1}.
\end{align*}
Since  $\hat{L}_t = G_+ L_0 G_+^{-1}$ it satisfies a standard Lax equation
\[
\frac{d\hat{L}_t}{dt}  =  \left[ \hat{L}_t^+, \hat{L}_t   \right]
\]
(see \emph{e.g.}  \cite[proof of Theorem~2.7]{CdSdS07}).
Hence, noting that $L_t = F_t \hat{L}_t F_t^{-1}$, we get
\[
\frac{dL_t}{dt} = \left[L_t^++A_0, L_t  \right].
\]
which is equation \eqref{eq.2.1}.
\end{proof}

\begin{proposition}
\label{prop.2.8}
Equation \eqref{eq.2.10} is equivalent to the equation in $\Omega$
\begin{equation}
\label{eq.2.12}
\frac{dF_t}{dt} = F_t \hat{A}_0
\end{equation}
where $\hat{A}_0 = P_0 \hat{L}_t = P_0 G_+ L_0 G_+^{-1}$.
\end{proposition}

\begin{proof}
Let $A_0 = P_0 L_t$ where $P_0$ is a linear bounded operator as assumed in the definition of the right-hand side of equation \eqref{eq.2.1}.
Define
\begin{equation}
\label{eq.2.13}
\hat{A}_0 = F_t^{-1} A_0 F_t.
\end{equation}
Noting that $F_t$ is independent of $\lambda$, $F_t$ commutes with $P_0$ which leads to
\[
\hat{A}_0 = P_0 F_t^{-1} L_t F_t = P_0\hat{L}_t = P_0 G_+L_0G_+^{-1}.
\]
Introducing in equation \eqref{eq.2.10} the definition \eqref{eq.2.13} of $\hat{A}_0$ we obtain
\[
\frac{dF_t}{dt} = F_t \hat{A}_0 F_t^{-1} F_t = F_t \hat{A}_0,
\]
as required.
\end{proof}

\begin{remark}
\label{rmk.2.9}
Equation \eqref{eq.2.12} is more convenient for calculating the solution of equation \eqref{eq.2.1} since $\hat{A}_0$ is known
explicitly whereas $A_0$ is not.
\end{remark}

\begin{proposition}
\label{prop.2.9}
$G=\exp(tL_0)$ has a canonical factorization at a point $t_i\in\Omega$ with factors differentiable \emph{w.r.t.} $t$ in a neighborhood of $t_i$
\emph{iff} equation~\eqref{eq.2.1} has a solution at the point $t_i$.
\end{proposition}

\begin{proof}
\noindent\emph{Sufficiency:} Assume that equation~\eqref{eq.2.1} has a solution at a point $t_i\in\Omega$. Then
by Theorem~ \ref{thm.2.3} there exist functions $\widetilde{G}_-$, 
$\widetilde{G}_+$ satisfying \eqref{eq.2.5} in a neighborhood of $t_i$ ($\Omega$ is open) which give the solution to equation~\eqref{eq.2.1},
\[
L_t = \widetilde{G}_+ L_0 \widetilde{G}_+^{-1} = \widetilde{G}_-^{-1} L_0 \widetilde{G}_-.
\]
From Theorem~\ref{thm.2.6} $\widetilde{G}_+$, $\widetilde{G}_-$ are related to the factors of another
canonical factorization of $G$ ($G_-$, $G_+$) by the
formulas
\begin{align*}
G_- = \widetilde{G}_- F_t, && G_+ = F_t^{-1}\widetilde{G}_+
\end{align*}
where $F_t$ satisfies the differential equation \eqref{eq.2.10}. Since $ \widetilde{G}_- $, $ \widetilde{G}_+$  and $F_t$ are differentiable in a vicinity
of $t_i$ it follows that the factors $G_-$, $G_+$ are differentiable too.

\smallskip

\noindent\emph{Necessity:} Assume that the factors $G_-$, $G_+$  of a canonical factorization of $G$ are differentiable. Then
\[
\frac{dG}{dt} = L_0 G = \frac{dG_-}{dt}\, G_+ + G_-\, \frac{dG_+}{dt}
\]
and letting $\hat{L}_t = G_-^{-1} L_0 G_-$ we get from the above relation
\[
\hat{L}_t^+ := P^+\hat{L}_t = \frac{dG_+}{dt}\, G_+^{-1},
\]
and, putting $L_t = F_t \hat{L}_t F_t^{-1}$ with $F_t$ satisfying \eqref{eq.2.10}, we have (see the proof of Theorem~\ref{thm.2.7})
\[
\frac{dL_t}{dt} = \left[L_t^+ + A_0, L_t    \right]
\]
which is equation \eqref{eq.2.1}.
\end{proof}


\section{Singularities via the Riemann-Hilbert approach}
\label{sect.3}

In this section we 
present our main result which enables us to locate the singularities of the solution to equation~\eqref{eq.2.1} without
obtaining the explicit solution of the associated Riemann-Hilbert problem. Here the use of the factorization method is crucial
since it allows us to translate the problem of the existence of singularities into an operator theory problem.

We recall from the Introduction the direct sum decomposition of $\left[C_\mu\left(S^1\right) \right]^n$,
\begin{equation}
\label{eq.dec.1}
\left[C_\mu\left(S^1\right) \right]^n = \left[C_\mu^+ \right]^n \oplus \left[C_\mu^- \right]^n_0,
\end{equation}
where  $\left[C_\mu^+\right]^n$ denotes the subspace of $\left[C_\mu\left(S^1\right)\right]^n$ of functions analytic in 
$\bbd$ and 
$\left[C_\mu^-\right]^n_0$ is the subspace of   analytic functions in 
$\bbc\setminus\overline{\bbd}$ that vanish at infinity. We let $P^+\colon \left[C_\mu\left(S^1\right) \right]^n \to \left[C_\mu^+\right]^n$ denote the projection associated to this decomposition (so that $\ker P^+ = \left[C_\mu^-\right]^n_0$).

Given a matrix $G\in \left[C_\mu\left(S^1\right) \right]^{n\times n}$ the corresponding multiplication operator  
in $\left[C_\mu^+ \right]^n$ is denoted $GI_+$. The composite 
\begin{equation*}
\label{eq. Top.op.1}
P^+GI_+\colon \left[C_\mu^+ \right]^n \to \left[C_\mu^+ \right]^n
\end{equation*}
is a Topelitz operator (with symbol $G$  \cite[Ch.1]{ClGo81}) whose properties are closely related to those of its 
symbol $G$. We recall from  \cite[Ch.1]{ClGo81} that the operator $P^+GI_+$
is invertible \emph{iff} $G$ has a canonical Wiener-Hopf factorization 
\[
G = G_- G_+,
\] 
with $(G^\pm)^{\pm 1} \in \left[C_\mu^\pm\left(S^1\right) \right]$. 

We are now ready to state the main result of this section.

\begin{theorem}
\label{thm.3.1}
Let $T_G$ be the Toeplitz operator $P^+GI_+$ defined above,
where $G=\exp(tL_0)$.

Then equation~\eqref{eq.2.1} has a solution in a neighborhood of a point $t=t_i$ \emph{iff} the operator $T_G$ is injective
at that point, \emph{i.e.}, $\ker T_G$ is trivial.
\end{theorem}
\begin{proof}
\noindent\emph{Sufficiency:}  We begin by proving that if $\ker T_G$ is trivial $T_G$ is invertible. Firstly we note that
 $G=\exp(tL_0)\in C_\mu\left(S^1\right)$ (in fact $G\in C^\infty\left(S^1\right)$) for every $t\in\bbc$. Thus a factorization of $G$ of the general form
\[
G = G_- D G_+, \qquad \left(G_\pm\right)^{\pm1} \in \left[ C_\mu^\pm\left(S^1\right)\right]^{n\times n},
\]
where $D$ is a diagonal nonsingular rational matrix, exists for all $t\in\bbc$ (see \emph{e.g.}  \cite[Ch.1]{ClGo81}).

Since $G\in \left[ C_\mu\left(S^1\right)\right]^{n\times n}$, $\det G\in C_\mu\left(S^1\right)$ and
\[
\det G = \exp((\operatorname{tr}L_0)t)\neq 0,
\]
for all $\lambda\in S^1$, it follows that $T_G$ is Fredholm of zero index.  This means that
\[
\operatorname{codim} \operatorname{im} T_G = \dim \ker T_G.
\]
It follows that, if $\ker T_G$ is trivial at $t=t_i$, $T_G$ is invertible at this point. It is easy to see that this is true
in a neighborhood of $t_i$. Hence $G$ has a canonical factorization
\[
G=\exp(tL_0) = G_-G_+
\]
in a neighborhood of $t_i$. By Proposition~\ref{prop.2.9} this implies that equation~\eqref{eq.2.1} has a solution at this point.

\smallskip

\noindent\emph{Necessity:} Assume that a solution to equation~\eqref{eq.2.1} exists at $t=t_i$. Then, by Proposition~\ref{prop.2.9},
$G$ possesses a canonical factorization at $t=t_i$, \emph{i.e.},
\[
G = G_- G_+, \qquad \left( G_\pm\right)^{\pm 1}\in \left[  C_\mu^\pm\left( S^1\right)\right]^{n\times n}.
\]
This is equivalent to the invertibility of $T_G$ in a neighborhood of $t_i$ and thus $\ker T_G$ is trivial.
\end{proof}

In the next two  propositions we express the condition of Theorem~\ref{thm.3.1} in terms of the existence of solutions to a certain Riemann-Hilbert problem, which has the advantage of being easier to analyse. 

\begin{proposition}
\label{prop.3.2}
Let $T_G$ be the operator defined in Theorem~\ref{thm.3.1}. Then $\ker T_G$ is nontrivial \emph{iff} the Riemann-Hilbert problem
\[
G\Phi^+ = \Phi^-, \qquad \Phi^\pm\in \left[ C_\mu^\pm (S^1) \right]^n,
\]
with $\Phi^-(\infty)=0$, has non-trivial solutions. 
\end{proposition}

\begin{proof}
$\ker T_G$ being non-trivial means that the equation
\[
P^+G\Phi^+ = 0, \qquad \Phi^+\in \left[ C_\mu^+ \right]^n
\]
has non-trivial solutions. Keeping in mind the direct sum decomposition \eqref{eq.dec.1}, we see that
this is equivalent to saying that the Riemann-Hilbert problem in $\left[C_\mu\left(S^1\right)\right]^n$
\[
G\Phi^+ = \Phi^-\quad \text{with} \quad \Phi^-(\infty)=0,
\]
has non-trivial solutions. 
\end{proof}

\begin{proposition}
\label{prop.3.3}
Let $n=2$ in Proposition~\ref{prop.3.2}. Then the vector valued  Riemann-Hilbert problem (on the Riemann sphere)
\begin{equation}
\label{eq.3.1}
G\Phi^+ = \Phi^-, \quad \Phi^-(\infty)=0,
\end{equation}
given in Proposition~\ref{prop.3.2} is equivalent to a scalar Riemann-Hilbert problem of the form
\begin{equation}
\label{eq.3.2}
g \Psi^+ = \Psi^-
\end{equation}
on a compact Riemann surface $\varSigma$ defined by the equation $\det(\mu I_2 - L(\lambda))=0$ 
 with $\Psi^-$ subject to the condition
\begin{equation}
\label{eq.3.3}
\Psi^-(\infty_1)=0,\, \Psi(\infty_2)=0
\end{equation}
where $\infty_1$, $\infty_2$ 
are the poles  of the meromorphic function given by the projection
\[
\varSigma \to \bbp^1(\bbc), (x, w)\mapsto x
\]
(\emph{i.e.}, $\infty_1$, $\infty_2$ are the points of $\varSigma$  "at infinity").
\end{proposition}
\begin{proof}
It is proven in \cite{CdSdS08} that the Riemann-Hilbert problem \eqref{eq.3.1} is equivalent, for $n=2$, to a scalar Riemann-Hilbert
problem on $\varSigma$  \eqref{eq.3.2}. The condition \eqref{eq.3.3} is the translation of the condition $\Phi^-(\infty)=0$ in
\eqref{eq.3.1} to the Riemann surface.
\end{proof}

\section{Example}
\label{sect.4}
In this section we study a dynamical system for which the solution and, consequently, its singularities 
can be obtained by classical methods and compare the result obtained with that given by the method of Section~\ref{sect.3}.

\subsection{Dynamical system}
\label{subsect.4.1}

We take the example presented in \cite{CdSdS07} which is given by the equations
\begin{equation}
\label{eq.4.1}
\frac{dL_t}{dt} = \left[ L_t^+, L_t\right]
\end{equation}
where
\begin{equation}
\label{eq.4.2}
L_t(\lambda) = \begin{bmatrix} v(\lambda) & \phantom{-}u(\lambda)\\ w(\lambda) &-v(\lambda)\end{bmatrix} , \qquad \lambda\in S^1
\end{equation}
with
\begin{align}
\label{eq.4.3}
\notag v(\lambda) & = z\lambda^{-1}\\
u(\lambda) & = a\lambda + y\lambda^{-1}+x, \quad a\in\bbc\\
\notag w(\lambda) & = a\lambda + y\lambda^{-1}-x
\end{align}
and $L_t^+$ being the polynomial part of $L_t$ (with respect to $\lambda$). It can  easily be seen that equation \eqref{eq.4.1} together
with \eqref{eq.4.2} and \eqref{eq.4.3} is equivalent to the following nonlinear system of differential equations
\begin{align}
\label{eq.4.4}
&\frac{dx}{dt} = -2az, && \frac{dy}{dt} = -2xz, && \frac{dz}{dt} = 2xy&
\end{align}
for the dynamical variables $x, y, z$. This system admits two integrals of the motion, namely,
\begin{align}
\label{eq.4.5}
A= x^2-2ay, && B=y^2+z^2.
\end{align} 
That these are invariants is easily checked by differentiating both sides of relations \eqref{eq.4.5} and using equations \eqref{eq.4.4}. 

\subsection{Classical solution}
\label{subsect.4.2}
To obtain an equation of the movement in the variable $x$ we begin with the first of equations \eqref{eq.4.4} 
\[
\left(\dot{x}\right)^2 = 4a^2z^2 = 4a^2(B-y^2)
\]
where $\dot{x}=\tfrac{dx}{dt}$. Using relations \eqref{eq.4.5} yields
\[
\left(\dot{x}\right)^2 = 4a^2B - (A-x^2)^2 = 4a^2B - A^2 + 2Ax^2-x^4
\]
from which we get
\begin{equation}
\label{eq.4.6}
\dot{x} = \mbi \sqrt{p(x)}
\end{equation}
where $p(x)=x^4-2Ax^2+A^2-4a^2B$. The above equation means that $(\dot{x}, x)$ lies in an elliptic curve, \emph{i.e.}, the orbits
of the dynamics lie in an elliptic Riemann surface.

Before we integrate \eqref{eq.4.6} we note that if we derive equations for the variables $y, z$ we obtain equation \eqref{eq.4.6}
after an elementary transformation on these variables as was to be expected.

Integration of \eqref{eq.4.6} gives
\begin{equation}
\label{eq.4.7}
\int_{x_0}^x \frac{dx}{\sqrt{p(x)}} = \mbi t 
\end{equation}
where $x_0$ is the value of $x$ at $t=0$, and the path of integration is understood to be on  the Riemann surface $\varSigma$
defined by
\begin{equation}
\label{eq.4.8}
w^2 = p(x) = (x^2-x_1^2)(x^2-x_2^2)
\end{equation}
with the zeros of $p(x)$, $\pm x_1$, $\pm x_2$, given by
\begin{align}
\label{eq.4.9}
&x_1^2= A+2a\sqrt{B},&&  x_2^2 = A-2a\sqrt{B}.&
\end{align}
It is useful to write \eqref{eq.4.8} in the normalized form
\[
w^2 = x_1^2x_2^2(1-{\widetilde{x}}^2)(1-k^2{\widetilde{x}}^2)
\]
where $\widetilde{x}=x/x_1$ and $k^2= x_1^2/x_2^2$. From now on we take as a definition of the Riemann surface $\varSigma$
the following equation
\begin{equation}
\label{eq.4.10'}
w^2=\left(1-x^2\right)\left(1-k^2x^2\right),
\end{equation}
which corresponds to making the change of variables $x\mapsto x/x_1$, $w\mapsto w/(x_1x_2)$. With this notation, \eqref{eq.4.7} takes the form
\[
\int_{\widetilde{x}_0}^{\widetilde{x}}\frac{dx}{w(x)} = \mbi x_2 t  \qquad \left(\text{on
 $\varSigma$\, \footnotemark} \right)
\]
\footnotetext{We identity $\varSigma$ with the quotient of $\bbc$ by the lattice of periods of  $dx/\sqrt{p(x)}$.}
with $\widetilde{x}=x/x_1$, $\widetilde{x}_0=x_0/x_1$. It is convenient to write the above integral as a difference of two 
integrals  as follows
\begin{equation}
\label{eq.4.10}
\mbi x_2 t = \int_0^{\widetilde{x}} \frac{dx}{w(x)} -\int_0^{\widetilde{x}_0}\frac{dx}{w(x)} \qquad \left(\text{on $\varSigma$}\right)
\end{equation}

We are looking for the singularities of the solution to equation \eqref{eq.4.6} so we let $x\to\infty$ which leads to
\begin{align*}
\int_0^{\widetilde{x}}\frac{dx}{w(x)} \to \mbi\mbK'&& \text{or}  &&\int_0^{\widetilde{x}}\frac{dx}{w(x)} \to \mbi\mbK'+2\mbK,
\end{align*}
where $\mbK$, $\mbK'$ are, respectively, the complete elliptic integral and the complementary complete elliptic integral of the first kind
(see \emph{e.g.} \cite{A90}). Hence from \eqref{eq.4.10}, keeping in mind that \eqref{eq.4.10} is an equation on $\varSigma$, we
obtain
\begin{align*}
\mbi x_2t& = \mbi\mbK' -u_0 + 4m\mbK + 2\mbi\mbK',
\end{align*}
or 
\begin{align*}
\mbi x_2t &= \mbi\mbK' + 2\mbK -u_0 + 4m\mbK + 2\mbi n \mbK', & n, m \in \bbz,
\end{align*}
where
\begin{equation}
\label{eq.*}
u_0 = \int_0^{\widetilde{x}_0}\frac{dx}{w(x)}.
\end{equation}

The above formulas for $t$ are equivalent to the single formula
\begin{align}
\label{eq.4.11}
\mbi x_2 t &= -u_0 + \mbi\mbK' + 2m\mbK+2\mbi n\mbK', & n, m \in \bbz.
\end{align}
This relation gives us the values of $t$ at which singularities occur, \emph{i.e.}, where the solution blows up.

\subsection{Riemann-Hilbert solution}
\label{subsect.4.3}
Next we derive a formula for the singularities of the solution to system \eqref{eq.4.4} using the method of Propositions~\ref{prop.3.2}
and \ref{prop.3.3}.
To this end we need to formulate the Riemann-Hilbert problem \eqref{eq.3.1} 
for the function $G=\exp(tL_0)$ in an associated Riemann surface. Recalling \eqref{eq.3.1}
we have
\begin{align}
\label{eq.4.12}
&G\Phi^+ = \Phi^- &&  \text{with} && G= \exp(tL_0),&
\end{align}
and $\Phi^\pm\in\left[C_\mu^\pm\right]^2$ with the condition $\Phi^-(\infty)=0$.

Taking into account that $L_0$ can be diagonalized  as
\begin{equation}
\label{eq.4.13}
L_0 = SD_0S^{-1},
\end{equation}
with
\begin{align}
\label{eq.4.14}
S=\begin{bmatrix} 1& -1\\ \frac{\mu-z_0}{q_1(\lambda)} & \frac{\mu+z_0}{q_1(\lambda)}\end{bmatrix} && 
\left(\,  q_1(\lambda):=a\lambda^2+x_0\lambda+y_0\, \right)
\end{align}
and $D_0=\diag(\lambda^{-1}\mu, -\lambda^{-1}\mu)$ where $\mu=\lambda\nu$ with $\nu$ given by the characteristic equation of $L_0$,
\begin{equation}
\label{eq.4.15}
\det\left( \nu I_2 - L_0(\lambda)\right) = 0.
\end{equation}

From this equation we obtain
\begin{equation}
\label{eq.4.16}
\mu^2 = p_1(\lambda):= a^2\lambda^4-(x_0^2-2ay_0)\lambda^2+z_0^2+y_0^2
\end{equation}
or, introducing the invariants $A$ and $B$,
\begin{equation}
\label{eq.4.17}
p_1(\lambda) = a^2\lambda^4 - A\lambda^2 + B.
\end{equation}

The explicit formulas for the zeros of $p_1(\lambda)$, $\pm\lambda_1, \pm\lambda_2$, are given in \eqref{eq.4.28} below.
Relation \eqref{eq.4.16} defines an elliptic  Riemann surface, which is associated  with $L_0$
(or $L_t$ as it is independent of the dynamics).  We denote by $\varSigma_1$ the \emph{compact
Riemann surface} obtained by adding two points at infinity $\infty_1, \infty_2$.

Going back to \eqref{eq.4.13} it follows from it that
\[
G = \exp\left( tL_0 \right) = SDS^{-1}
\]
where $D = \diag ( \exp(t\lambda^{-1}\mu), \exp(-t\lambda^{-1}\mu) )$.

Hence \eqref{eq.4.12} may be written as
\[
DS^{-1}\Phi^+ = S^{-1}\Phi^-,
\]
which, in terms of the components of $\Phi^\pm$, denoted $(\phi_1^\pm, \phi_2^\pm)$,  is written  as 
\begin{equation}
\label{eq.4.18}
\begin{cases}
d_1 \left (z_0\phi_1^+ + q_1\phi_2^++ \mu\phi_1^+\right) & = z_0\phi_1^-+q_1\phi_2^-+\mu\phi_1^-\\
d_2 \left (z_0\phi_1^+ + q_1\phi_2^+- \mu\phi_1^+\right) & = z_0\phi_1^-+q_1\phi_2^--\mu\phi_1^-
\end{cases}
\end{equation}
where $d_1 = \exp (t\lambda^{-1}\mu)$, $d_2 = \exp(-t\lambda^{-1}\mu)$.

The above system is equivalent to  the following single scalar equation (for more details see 
\cite{CdSdS07} or \cite{CdSdS08}) on a contour $\Gamma$ that is the preimage of $S^1$ under the projection
$\varrho\colon(\lambda, \mu)\mapsto\lambda$,
\begin{equation}
\label{eq.4.19}
d\left( \phi_2^+ + \frac{z_0+\mu}{q_1}\phi_1^+\right) = \phi_2^- + \frac{z_0+\mu}{q_1}\phi_1^-.
\end{equation}
Note that $\Gamma$ has two connected components; we put $d=d_1$ on one of these components and $d=d_2$
on the other. In view of the expressions for $d_1$ and $d_2$ we have
\begin{align}
\label{eq.4.20}
d &= \exp \left(\frac{\mu}{\lambda }\,t \right), && (\lambda, \mu)\in \Gamma. &
\end{align}

Concerning equation \eqref{eq.4.19}, it is also  useful to note that setting $q_2(\lambda)=a\lambda_2 - x_0\lambda +y_0$,
 we have
\begin{equation}
\label{eq.def.q1-q2}
\mu^2 -z_0^2 = q_1(\lambda)q_2(\lambda).
\end{equation}
It follows that,  as a meromorphic function on $\varSigma_1$, $q_1(\lambda)$ has four zeros,
two of which are  zeros of $\mu+z_0$ and the other two are zeros of $\mu-z_0$.

To solve \eqref{eq.4.19} we note that $d$ can be factorized on the Riemann surface as
\[
d = d_- r d_+,
\]
where $\left(d^+\right)^{\pm 1}\in C_\mu\left(\Gamma\right)$ extends holomorphically to the preimage $\Omega^+$  of $\bbd$  under the  projection $\varrho$ and, similarly,  $\left(d^-\right)^{\pm 1}\in C_\mu\left(\Gamma\right)$ admits a holomorphic extension to the preimage $\Omega^-$ of   $\bbp(\bbc)\setminus\overline{\bbd}$. Finally, $r$ is a rational function on $\varSigma_1$. See \cite{CdSdS07} or \cite{CdSdS08} for more details.

Note that all three factors in the above factorization depend on $t$. Introducing this factorization in \eqref{eq.4.19}, we get
\begin{align}
\label{eq.4.21}
rd_+\left( \phi_2^++ \frac{z_0+\mu}{q_1}\phi_1^+\right) & = d_-^{-1} \left( \phi_2^- + \frac{z_0+\mu}{q_1}\phi_1^-\right) = R,
\end{align}
where $R$ is a rational function on $\varSigma_1$. 

For the computations that follow it is
convenient to rewrite the   Riemann surface  $\varSigma_1$ using the normalized equation
\begin{equation}
\label{eq.def.of.varsigma_1}
\mu^2 =  (1-\lambda^2)(1-k_1^2\lambda^2),
\end{equation}
where  $k_1=\lambda_1/\lambda_2$ and
$\pm\lambda_1, \pm\lambda_2$ are the roots of $p_1(\lambda)$ (\emph{cf.} \eqref{eq.4.16})
given in  \eqref{eq.4.28} below.  
This corresponds to making a  change of variables $\lambda \mapsto \lambda/\lambda_1$, 
$ \mu \mapsto  \mu/(a\lambda_1\lambda_2)$.

Also, from now on we identify $\varSigma_1$ with its \emph{Jacobian}, using the 
Abel map
\begin{equation}
\label{eq.4.22}
(\lambda, \mu) \mapsto u = \int_0^{(\lambda, \mu)} \frac{d\lambda}{\mu},
\end{equation}
\emph{i.e.},  we consider all equations relating points of $\varSigma_1$ as written on the 
quotient of $\bbc$ by the lattice of periods of  the holomorphic form $d\lambda/\mu$. 

Hence, keeping in mind that $\varSigma_1$ is an elliptic Riemann surface ($p_1$ is a fourth degree polynomial), $R$ can be expressed in 
elliptic theta functions. To this end we recall that we are solving \eqref{eq.4.21} with the conditions $\phi_i^-(\infty_j)=0$,
for $i, j=1,2$, where $\infty_1, \infty_2$ are the two points at infinity
\footnote{We choose $\infty_1$ such that  $\mu\sim k_1\lambda^2$ near $\infty_1$. }
 in $\varSigma_1$, which correspond to $\infty$ under the
projection $\varrho\colon(\lambda, \mu) \mapsto \lambda$. Denoting by $\psi^+, \psi^-$ the expression within parentheses in both sides of \eqref{eq.4.21}, these conditions
correspond to
\begin{align}
\label{eq.4.23}
&\psi^-(\infty_1)=0, & \psi^-(\infty_2)=0.&
\end{align}
Before introducing these conditions we note that, using the Jacobi theta function $\vartheta_1$ that satisfies $\vartheta_1(0)=0$, $R$
has the expression
\begin{equation}
\label{eq.4.24}
R(u) = \gamma\, \frac{\vartheta_1(u-v_0)\vartheta_1(u-v_1)\vartheta_1(u-v_2)}{\vartheta_1(u-u_0)\vartheta_1(u-u_1)\vartheta_1(u-u_2)}
\end{equation}
where $\gamma\in\bbc$ and the zeros and poles of $R$ are determined by the following conditions:
\begin{enumerate}[(i)]
\item $R$ has a pole  at the point $u_0$ corresponding to the pole of $r$ in $\Omega^+$ (see \cite[Appendix~B]{CdSdS07});
\item $R$ has a zero at the point $v_0$ corresponding to the zero of $r$ in $\Omega^+$ (see \cite[Appendix~B]{CdSdS07});
\item $R$ has two poles $u_1, u_2$ at the zeros of $q_1$ that do not coincide with the zeros of $a\lambda_1\lambda_2\mu+z_0$
(which is $\mu + z_0$ in \eqref{eq.4.21}  written in the normalized coordinates of \eqref{eq.def.of.varsigma_1});
\item $R$ has two zeros at points $v_1, v_2$ imposed by condition \eqref{eq.4.23}, \emph{i.e.},
\begin{align*}
v_1=\infty_1,\quad v_2=\infty_2.
\end{align*}
\item The zeros and poles of $R$ must satisfy Abel's condition:
\begin{align}
\label{eq.4.25}
&v_0-u_0 = u_1 + u_2-\infty_1 - \infty_2  & (\operatorname{mod} 2\mbi\mbK_1' + 4\mbK_1),&
\end{align}
where $\mbK_1$ and $\mbK_1'$ are the complete elliptic and complementary elliptic integrals of the first kind of $\varSigma_1$. 
\end{enumerate}

From the analysis of the factorization of  the function $d$ given in \eqref{eq.4.20} (see \cite[Definition~B.7 and Proposition~B.9]{CdSdS07}) we obtain
\begin{equation}
\label{eq.4.26'}
v_0-u_0 = 2at\lambda_2
\end{equation}
where $\lambda_2$ is as in the expression for $p_1(\lambda)$ (see text following \eqref{eq.4.17}). Substitution of \eqref{eq.4.26'} in \eqref{eq.4.25} gives us the expression for the values
of $t$ for which singularities occur. Taking into account that $\infty_1+\infty_2=2\mbK_1 + 2\mbi\mbK_1'$ 
$(\operatorname{mod}\,  4m\mbK_1 + 2\mbi n \mbK_1')$ 
we have
\begin{equation}
\label{eq.4.26}
2at\lambda_2 = u_1 + u_2 + 2\mbK_1 + 4m\mbK_1 + 2\mbi n \mbK_1'
\end{equation}
where $u_1, u_2$ are the images under Abel's map of the zeros of $q_1$  that do not coincide with zeros of 
$a\lambda_1\lambda_2\mu+z_0$, \emph{i.e.},
\begin{align*}
&u_1 = \int_0^{\hat{\lambda}_1/\lambda_1}\frac{d\lambda}{\mu}, & u_2 = \int_0^{\hat{\lambda}_2/\lambda_1}\frac{d\lambda}{\mu},&
\end{align*}
where $\hat{\lambda}_1, \hat{\lambda}_2$ are de zeros of $q_1(\lambda)$ and $\lambda_1$ is a zero of $p_1(\lambda)$ given in \eqref{eq.4.28} below.

Thus \eqref{eq.4.26} gives us the values of $t$ leading to singularities of the solution of Lax equation \eqref{eq.2.1} as derived from the theory of
Section~\ref{sect.3}.

\begin{remark}
\label{rmk.4.1}
\begin{enumerate}[(i)]
\item Formula \eqref{eq.4.26} was obtained without requiring an explicit formula for $\phi_1^\pm$, $\phi_2^\pm$ 
corresponding to the factors of the canonical factorization of $G$, $G=G_-G_+$,  although these 
functions can easily be obtained from \eqref{eq.4.21}, 
replacing condition {\it (iv)} by the imposition of a zero at a chosen point $v_1$. Then Abel's condition {\it (v)} gives the zero $v_2$
(see \cite{CdSdS07} for the details). The solution thus obtained gives the factors $G_-$, $G_+$ of $G$ providing
$t$ does not satisfy \eqref{eq.4.26}, a result that could not be obtained in \cite{CdSdS07}.

\item Formulas \eqref{eq.4.26}  and  \eqref{eq.4.11} are not easily compared since they involve different
Riemann surfaces. The   appearance of distinct surfaces when using different methods to study integrable systems is an intriguing phenomenon  that occurs in  other examples \cite{AU96}.
\end{enumerate}
\end{remark}

We show next that the two Riemann surfaces are closely related and that the two expressions for the singularities coincide.

\subsection{Comparison of solutions}
\label{subsect.4.4}
We start by  showing  that the two Riemann surfaces $\varSigma$ and $\varSigma_1$ are related and this will enable us to express
both formulas \eqref{eq.4.11} and \eqref{eq.4.26} on the same Riemann surface thus allowing for a comparison of the two results.

The Riemann surface $\varSigma$ is defined by the equation
\[
w^2 = (1-x^2)(1-k^2x^2)=\frac{p(x_1x)}{(x_1x_2)^2}
\]
with $p(x)=x^4-2Ax^2+A^2-4a^2B = (x^2-x_1^2)(x^2-x_2^2)$ and  $k=x_1/x_2$, where
\begin{align}
\label{eq.4.27}
x_1^2 =   A+2a\sqrt{B},&&  x_2^2 = A- 2a\sqrt{B}.
\end{align}

The Riemann surface $\varSigma_1$ is defined by the equation
\[
\mu^2 = (1-\lambda^2)(1-k_1^2\lambda^2)=\frac{p_1(\lambda_1\lambda)}{(a\lambda_1\lambda_2)^2}
\]
with $p_1(\lambda)=a^2\lambda^4-A\lambda^2+B= a^2(\lambda^2-\lambda_1^2)(\lambda^2-\lambda_2^2)$, where 
\begin{align}
\label{eq.4.28}
\lambda_1^2= \frac{A+\sqrt{A^2-4a^2B}}{2a^2},&& \lambda_2^2 = \frac{A-\sqrt{A^2-4a^2B}}{2a^2}.
\end{align}
From the expression for $p(x)$ and \eqref{eq.4.27} we have
\begin{align*}
A^2-4a^2B & = x_1^2x_2^2\\
2A& = x_1^2 + x_2^2.
\end{align*}
Introducing these relations in \eqref{eq.4.28} gives
\begin{align}
\label{eq.4.29}
\lambda_1^2 = \left( \frac{x_2-x_1}{2a} \right)^2, &&\lambda_2^2 = \left( \frac{x_1+x_2}{2a} \right)^2
\end{align}
which leads to 
\begin{align}
\label{eq.4.30}
\lambda_1 =   \frac{x_2- x_1}{2a}, &&  \lambda_2 =  \frac{x_1+ x_2}{2a} 
\end{align}
where the sign in the square root is determined by a direct check on \eqref{eq.4.28}. From \eqref{eq.4.30}
we now get the relation between the moduli of the two surfaces 
\begin{equation}
\label{eq.4.31}
k_1 = \frac{\lambda_1}{\lambda_2} = \frac{x_2-x_1}{x_2+x_1} = \frac{1-k}{1+k}
\end{equation}
where $k$ and $k_1$ denote the \emph{elliptic moduli} of $\varSigma$ and $\varSigma_1$, respectively (see \cite{A90}). This
shows that the surfaces are closely related as claimed at the end of Section~\ref{subsect.4.3}.

Having obtained equality \eqref{eq.4.31} we are now in a position to state the following proposition relating
$\varSigma$ and $\varSigma_1$.

\begin{proposition}
\label{prop.4.2}
The following statements express the relation between the Riemann surfaces $\varSigma$ and $\varSigma_1$:
\begin{enumerate}[(i)]
\item For the  elliptic moduli of $\varSigma$ and $\varSigma_1$, respectively $k$, $k_1$, we have  
\[
k_1 = \frac{1-k}{1+k}
\]
\item 
There is a holomorphic map $\varphi\colon\varSigma\to\varSigma_1$ given by
\[
\varSigma \ni (x, w) \mapsto (\lambda, \mu) := \left(\mbi(1+k)\frac{x}{w}, \frac{k^2x^4-1}{w^2}    \right) \in \varSigma_1.
\]
\item Under the map $\varphi$ of {\it (ii)} the points at infinity of $\varSigma$ are mapped to  
$\mbo_1:=(0,1)$ and  $(0,\pm 1)$, 
$(\pm 1, 0)$, $(\pm 1/k, 0)$ are mapped to $\mbo_2:=(0,-1)$, $\infty_1$ and $\infty_2$, respectively.
\item The relation between the holomorphic forms of both surfaces is expressed by
\[
\varphi^*\left(\frac{d\lambda}{\mu} \right) = -\mbi (1+k) \frac{dx}{w}.
\]
\end{enumerate}
\end{proposition}

\begin{proof}
{\it (i)}  was proven in \eqref{eq.4.31}. 

The formula in {\it (ii)} is obtained  by composing the  two Gauss transformations corresponding 
in terms of elliptic moduli to $k\mapsto k'_1$ and $k'\mapsto k_1$ (see \cite[\S 39]{A90}). That it
defines a map  $\varSigma\to\varSigma_1$ can be directly checked by a substitution of \eqref{eq.4.31}  in  
$\mu^2 = (1-\lambda^2)(1-k_1^2\lambda^2)$. We note that this map is not injective; in fact 
it is $2$ to $1$.  {\it (iii)}  is easily obtained by direct substitution in formula~{\it (ii)}.

The expression {\it (iv)} follows directly from   {\it (ii)} by differentiation.
\end{proof}

Before we attempt to formulate expression \eqref{eq.4.11} in $\varSigma_1$ we are going to write $u_1+u_2$ of \eqref{eq.4.26} as
a single integral as in \eqref{eq.4.11} in order to make it possible to compare the two results. From \eqref{eq.4.26}
\[
u_1 + u_2 = \int_0^{\hat{\lambda}_1/\lambda_1} \frac{d\lambda}{\mu} + \int_0^{\hat{\lambda}_2/\lambda_1} \frac{d\lambda}{\mu}
\]
which we seek to write in the form
\begin{equation}
\label{eq.4.32}
u_1 + u_2 = \int_0^{\xi_0} \frac{d\lambda}{\mu}.
\end{equation}
To obtain $\xi_0$ we make use of the formula for the sum of arguments of the elliptic function $\sn$ (see \cite{A90}),
\begin{equation}
\label{eq.4.33}
\sn (u_1+u_2) = \frac{\sn u_1 \cn u_2 \dn u_2 + \sn u_2 \cn u_1 \dn u_1}{1-k_1^2 \sn^2 u_1 \sn^2 u_2}.
\end{equation}
We have:
\begin{align*}
\sn u_1 & = \frac{\hat{\lambda}_1}{\lambda_1},\qquad \sn u_2  = \frac{\hat{\lambda}_2}{\lambda_1},\\
\cn u_2 \dn u_2 &= \left[ 1-\left(\frac{\hat{\lambda}_2}{\lambda_1}\right)^2  \right]^{1/2}\cdot\left[1-k_1^2\left(\frac{\hat{\lambda}_2}{\lambda_1}\right)^2  \right]^{1/2}
= \mu \left(\frac{\hat{\lambda}_2}{\lambda_1} \right),\\
\cn u_1 \dn u_1 &= \mu\left(\frac{\hat{\lambda}_1}{\lambda_1} \right).
\end{align*}
Since $q_1(\lambda_1\lambda)q_2(\lambda_1\lambda) = (a\lambda_1\lambda_2\mu)^2-z_0^2$ (see \eqref{eq.def.q1-q2}) and $\hat{\lambda}_1$,
$\hat{\lambda}_2$ are the zeros of $q_1$ that are not  zeros of $a\lambda_1\lambda_2 + z_0$, we have
\begin{equation}
\label{eq.4.34}
\mu\left(\frac{\hat{\lambda}_1}{\lambda_1}\right) = \mu\left(\frac{\hat{\lambda}_2}{\lambda_1}\right)
 = \frac{z_0}{a\lambda_1\lambda_2}
\end{equation}
where the factor $1/(a\lambda_1\lambda_2)$ comes from the normalization of $\mu$.

Using the above results in \eqref{eq.4.33}, taking into account, \eqref{eq.4.32} gives 
\begin{equation}
\label{eq.4.35}
\xi_0 = \frac{1}{\lambda_1}\,\frac{\hat{\lambda}_1+\hat{\lambda}_2}{1-k_1^2\hat{\lambda}_1^2\hat{\lambda}_2^2/\lambda_1^4}\,\frac{z_0}{a\lambda_1\lambda_2}
\end{equation}
The denominator of the above formula can be simplified as follows
\[
1-k_1^2\,\frac{\hat{\lambda}_1^2\hat{\lambda}_2^2}{\lambda_1^4} = 1 - \frac{\hat{\lambda}_1^2\hat{\lambda}_2^2}{\lambda_1^2\lambda_2^2}
\]
since $k_1^2 = \lambda_1^2/\lambda_2^2$. Using \eqref{eq.4.34} we obtain
\[
1 - \frac{\hat{\lambda}_1^2\hat{\lambda}_2^2}{\lambda_1^2\lambda_2^2} = 1 -\frac{y_0^2}{a^2\lambda_1^2\lambda_2^2} = \frac{z_0^2}{B}
\]
as $\hat{\lambda}_1^2\hat{\lambda}_2^2=y_0^2/a^2$, $\lambda_1^2\lambda_2^2=B/a^2$ and $B=y_0^2+z_0^2$.

Finally,
\begin{equation}
\label{eq.4.36}
\xi_0 = \frac{x_0}{\lambda_1a}\, \frac{z_0}{a\lambda_1\lambda_2}\, \frac{B}{z_0^2} = \frac{x_0}{z_0} \lambda_2
\end{equation}
where we have used the result $\hat{\lambda}_1 + \hat{\lambda}_2 = x_0/a$.

We shall now transform the terms on the right-hand side of \eqref{eq.4.11} into the surface $\varSigma_1$. We first take the
expression for $u_0$  given in  \eqref{eq.*}
\begin{equation}
\label{eq.4.37}
u_0 = \int_0^{x_0/x_1} \frac{dx}{w(x)} = \frac{\mbi}{1+k} \int_0^{\lambda_0} \frac{d\lambda}{\mu(\lambda)},
\end{equation}
where
\begin{equation}
\label{eq.4.38}
\lambda_0 = \mbi (1+k)\, \frac{x_0/x_1}{w(x_0/x_1)}.
\end{equation}
For the sake of simplicity in the calculations instead of trying to transform $\lambda_0$ into $\xi_0$ we prefer to take $\xi_0$ in \eqref{eq.4.36}
and transform it as follows:
\begin{align} 
\label{eq.4.39}
\xi_0 & = \frac{x_0}{z_0}\, \lambda_2 = \frac{x_0}{w(x_0/x_1)}\, \frac{2\mbi a}{x_1x_2}\,\lambda_2\\
          & =   \frac{x_0}{w(x_0/x_1)}\, \frac{2\mbi a}{x_1x_2}\, \frac{x_1+x_2}{2a}   \notag \\
          &= \mbi\, \frac{x_0}{x_1}\, \frac{1}{w(x_0/x_1)}\, (1+k) = \lambda_0. \notag
\end{align}
where we have used  formula~\eqref{eq.4.40} and the relation $w(x_0/x_1) = $ $2\mbi az_0/(x_1x_2)$, which is a consequence
of \eqref{eq.4.4} and \eqref{eq.4.6} for $t=0$. 
\eqref{eq.4.30}. To end the calculation for the comparison of formulas \eqref{eq.4.11}
and \eqref{eq.4.26} we need to derive relations between the complete elliptic integrals on  $\varSigma$ and $\varSigma_1$. Using 
{\it (ii)} and {\it (iv)} of Proposition~\ref{prop.4.2} we have
\begin{align}
\label{eq.4.40}
\mbK &= \int_0^1 \frac{dx}{w(x)} = \frac{1}{\mbi(1+k)} \int_0^\infty \frac{d\lambda}{\mu(\lambda)} = \frac{\mbK_1'}{\mbi(1+k)}\\
\label{eq.4.41}
\mbK' &= \int_0^\infty \frac{dx}{w(x)} = \frac{1}{\mbi(1+k)} \int_{\mbo_2}^{\mbo_1} \frac{d\lambda}{\mu(\lambda)} = \frac{2\mbK_1}{\mbi(1+k)}. 
\end{align}

Substitution of \eqref{eq.4.39}, \eqref{eq.4.40} and \eqref{eq.4.41} in \eqref{eq.4.11} now gives
\begin{equation}
\label{eq.4.43}
-(1+k) x_2 t = -i(1+k)u_0 + 2\mbK_1 + 4n\mbK_1 + 2\mbi m K_1'.
\end{equation}
Using \eqref{eq.4.30} in \eqref{eq.4.26} and  the equalities $k=x_1/x_2$, $\lambda_2=(x_1+x_2)/2a$,
we see that formulas \eqref{eq.4.43} and \eqref{eq.4.26} coincide.

\bibliography{Singularities}
\bibliographystyle{amsplain}
\end{document}